\documentclass[a4paper, 10pt]{article}

\usepackage[latin1]{inputenc}
\usepackage{subfigure}
\usepackage{amsmath}
\usepackage{amssymb}
\usepackage{amsthm}
\usepackage{amsfonts}
\usepackage[cyr]{aeguill}
\usepackage{xspace}
\usepackage[english]{babel}
\usepackage{graphicx}
\usepackage{float}
\usepackage{mathtools}
\usepackage{psfrag,floatflt,hyperref}
\usepackage{verbatim}
\usepackage{geometry}
\usepackage{multirow}
\usepackage[retainorgcmds]{IEEEtrantools}
\usepackage{color}
\newtheorem{theorem}{Theorem}[section]

\newtheorem{proposition}[theorem]{Proposition}
\newtheorem{corollary}[theorem]{Corollary}

\newenvironment{remark}[1][Remark.]{\begin{trivlist}
\item[\hskip \labelsep {\bfseries #1}]}{\end{trivlist}}
\newenvironment{example}[1][Examples.]{\begin{trivlist}
\item[\hskip \labelsep {\bfseries #1}]}{\end{trivlist}}

\title{A formula for the number of spanning trees in circulant graphs with non-fixed generators and discrete tori\footnote{The author acknowledges support from the Swiss NSF grant $200021\_132528/1$.}}
\author{Justine Louis}
\date{15 December 2014}

\DeclareMathOperator{\argcosh}{Argcosh}

\begin{document}
        \maketitle

\begin{abstract}
We consider the number of spanning trees in circulant graphs of $\beta n$ vertices with generators depending linearly on $n$. The matrix tree theorem gives a closed formula of $\beta n$ factors, while we derive a formula of $\beta-1$ factors. Using the same trick, we also derive a formula for the number of spanning trees in discrete tori. Moreover, the spanning tree entropy of circulant graphs with fixed and non-fixed generators is compared.
\end{abstract}

\section{Introduction}
A spanning tree of a connected graph $G$ is a connected subgraph of $G$ without cycles with the same vertex set as $G$. The number of spanning trees in a graph $G$, $\tau(G)$, is an important graph invariant and is widely studied. It can be computed from the well-known matrix tree theorem due to Kirchhoff (e.g. see \cite{MR1271140}). Let $V(G)$ be the set of vertices of $G$ and $f:V(G)\rightarrow\mathbb{R}$ a function. The combinatorial Laplacian on $G$ is defined by
\begin{equation*}
\Delta_Gf(x)=\sum_{y\sim x}(f(x)-f(y))
\end{equation*}
where the sum is over all vertices adjacent to $x$. The matrix tree theorem states that
\begin{equation}
\label{kirchhoff}
\tau(G)=\frac{1}{\lvert V(G)\rvert}\textnormal{det}^\ast\Delta_G
\end{equation}
where $\textnormal{det}^\ast\Delta_G$ denotes the product of the non-zero eigenvalues of the Laplacian on $G$. In this paper we prove closed formulas for $\tau(G)$ for two types of graphs in terms of eigenvalues of the Laplacian on a subgraph of $G$. The formulas are particularly interesting when the number of vertices is larger than the other parameters of the graph.

Let $1\leqslant\gamma_1\leqslant\ldots\leqslant\gamma_d\leqslant\lfloor n/2\rfloor$ be positive integers. A circulant graph $C^{\gamma_1,\ldots,\gamma_d}_{n}$ is the $2d$-regular graph with $n$ vertices labelled $0,1,\ldots,n-1$ such that each vertex $v\in\mathbb{Z}/n\mathbb{Z}$ is connected to $v\pm\gamma_i$ mod $n$ for all $i\in\{1,\ldots,d\}$. The first type of graphs studied is the circulant graph with the first generator equal to one and the $d-1$ others linearly depending on the number of vertices, that is $C^{1,\gamma_1n,\ldots,\gamma_{d-1}n}_{\beta n}$, where $1\leqslant\gamma_1\leqslant\ldots\leqslant\gamma_{d-1}\leqslant\lfloor\beta/2\rfloor$ and $\beta$ are integers. Two examples are illustrated in Figure \ref{graph} below. It is known that the number of spanning trees in circulant graphs with $n$ vertices satisfies a linear recurrence relation with constant coefficients in $n$, this has been shown by Golin, Leung and Wang in \cite{golin2005counting}. For $\beta\in\{2,3,4,6,12\}$, closed formulas 
have been obtained by Zhang, Yong and Golin in \cite{zhang2005chebyshev} where the authors used techniques inspired from Boesch and Prodinger \cite{boesch1986spanning} using Chebyshev polynomials. As noted in \cite{zhang2005chebyshev} this method does not work for other values of $\beta$. In section \ref{th}, we derive Theorem \ref{ThCirc} in a simple way which gives a closed formula for all integer values of $\beta$. This gives an answer to an open question in \cite{golin2005unhooking} and \cite{zhang2005chebyshev} and proves the conjecture stated in \cite{louis2013asymptotics}. The second type of graphs studied is the $d$-dimensional discrete torus defined by the quotient $\mathbb{Z}^d/\Lambda\mathbb{Z}^d$, where $\Lambda$ is a diagonal integer matrix, with nearest neighbours connected. In the last section, we deduce the tree entropy for a sequence of non-fixed generated circulant graphs and compare it to the one with fixed 
generators.
\begin{figure}[!ht]
\label{graph}
\centering
\subfigure[$C^{1,n}_{5n}$ with $n=10$]{\includegraphics[width=5cm]{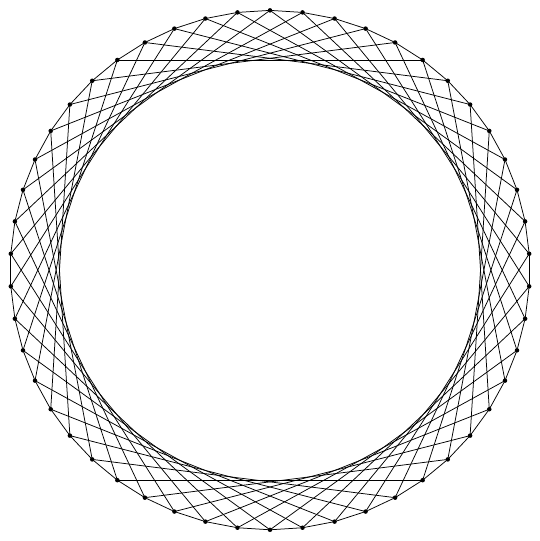}}
\hspace{2cm}
\subfigure[$C^{1,3n,4n}_{12n}$ with $n=4$]{\includegraphics[width=5cm]{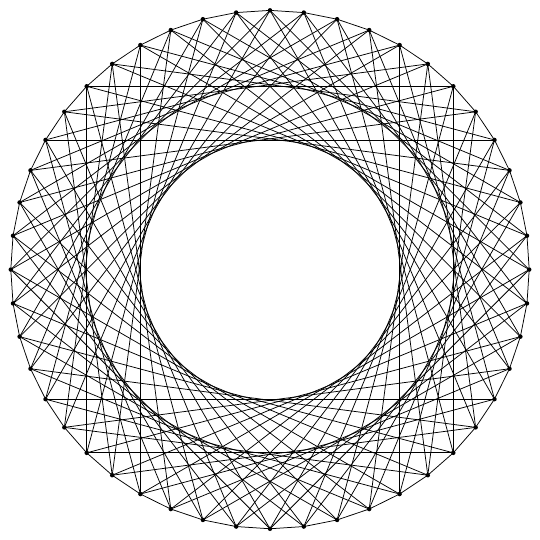}}
\caption{Examples of circulant graphs.}
\end{figure}
\par\vspace{\baselineskip}
\noindent
\textbf{Acknowledgements:} The author thanks Anders Karlsson for encouraging and helpful discussions and support.
\section{Spanning trees in circulant graphs with non-fixed generators}
\label{th}
Since the circulant graph $C^{1,\gamma_1n,\ldots,\gamma_{d-1}n}_{\beta n}$ is the Cayley graph of the group $\mathbb{Z}/\beta n\mathbb{Z}$, the eigenvectors of the Laplacian are given by the characters
\begin{equation*}
 \chi_k(x)=e^{2\pi ikx/(\beta n)},\quad k=0,1,\ldots,\beta n-1.
\end{equation*}
Therefore the eigenvalues are given by
\begin{equation}
\label{ev}
\lambda_k=2d-2\cos(2\pi k/(\beta n))-2\sum_{m=1}^{d-1}\cos(2\pi k\gamma_m/\beta),\quad k=0,1,\ldots,\beta n-1.
\end{equation}
\begin{theorem}
\label{ThCirc}
Let $1\leqslant\gamma_1\leqslant\ldots\leqslant\gamma_{d-1}\leqslant\lfloor\beta/2\rfloor$ be positive integers and $\mu_k=2(d-1)-2\sum_{m=1}^{d-1}\cos(2\pi k\gamma_m/\beta)$, $k=1,\ldots,\beta-1$, be the non-zero eigenvalues of the Laplacian on the circulant graph $C^{\gamma_1,\ldots,\gamma_{d-1}}_\beta$. For all $n\in\mathbb{N}_{\geqslant1}$, the number of spanning trees in the circulant graph $C^{1,\gamma_1n,\ldots,\gamma_{d-1}n}_{\beta n}$ is given by
\begin{equation*}
\tau(C^{1,\gamma_1n,\ldots,\gamma_{d-1}n}_{\beta n})=\frac{n}{\beta}\prod_{k=1}^{\beta-1}\big((\mu_k/2+1+\sqrt{\mu_k^2/4+\mu_k})^n+(\mu_k/2+1-\sqrt{\mu_k^2/4+\mu_k})^n-2\cos(2\pi k/\beta)\big).
\end{equation*}
\end{theorem}
\begin{remark}
It would be interesting to see if this pattern appears in other types of graphs, that is, the number of spanning trees could be expressed in terms of the eigenvalues of the Laplacian on a subgraph of the original graph.
\end{remark}
\begin{proof}
Applying the matrix tree theorem (\ref{kirchhoff}) to the graph $C^{1,\gamma_1n,\ldots,\gamma_{d-1}n}_{\beta n}$, with eigenvalues given by (\ref{ev}), gives
\begin{equation*}
\tau(C^{1,\gamma_1n,\ldots,\gamma_{d-1}n}_{\beta n})=\frac{1}{\beta n}\prod_{k=1}^{\beta n-1}\bigg(2d-2\cos(2\pi k/(\beta n))-2\sum_{m=1}^{d-1}\cos(2\pi k\gamma_m/\beta)\bigg).
\end{equation*}
Since there are $n$ spanning trees in the cycle $C^1_n$, that is,
\begin{equation}
\label{cycle}
n=\tau(C^1_n)=\frac{1}{n}\prod_{k=1}^{n-1}(2-2\cos(2\pi k/n)),
\end{equation}
it follows that
\begin{align}
&\tau(C^{1,\gamma_1n,\ldots,\gamma_{d-1}n}_{\beta n})=\frac{n}{\beta}\prod_{\substack{k=1\\ \beta\nmid k}}^{\beta n-1}\bigg(2d-2\cos(2\pi k/(\beta n))-2\sum_{m=1}^{d-1}\cos(2\pi k\gamma_m/\beta)\bigg)\nonumber\\
&=\frac{n}{\beta}\prod_{k=1}^{\beta-1}\prod_{l=0}^{n-1}\bigg(2d-2\cos(2\pi(k+l\beta)/(\beta n))-2\sum_{m=1}^{d-1}\cos(2\pi(k+l\beta)\gamma_m/\beta)\bigg)\nonumber\\
&=\frac{n}{\beta}\prod_{k=1}^{\beta-1}\prod_{l=0}^{n-1}\bigg(2\cosh(\argcosh(d-\sum_{m=1}^{d-1}\cos(2\pi k\gamma_m/\beta)))-2\cos(2\pi k/(\beta n)+2\pi l/n)\bigg).
\label{cosh}
\end{align}
We now evaluate the product over $l$ by the following calculation
\begin{align}
\prod_{l=0}^{n-1}(2\cosh\theta-2\cos((\omega+2\pi l)/n))&=e^{-n\theta}\prod_{l=0}^{n-1}(e^{2\theta}-2\cos((\omega+2\pi l)/n)e^\theta+1)\nonumber\\
&=e^{-n\theta}\prod_{l=0}^{n-1}(e^\theta-e^{i(\omega+2\pi l)/n})(e^\theta-e^{-i(\omega+2\pi l)/n}).
\label{prod}
\end{align}
The complex numbers $e^{i(\omega+2\pi l)/n}$ and $e^{-i(\omega+2\pi l)/n}$ for $l=0,1,\ldots,n-1$ are the $2n$ roots of the following polynomial in $e^\theta$
\begin{equation*}
e^{2n\theta}-2e^{n\theta}\cos\omega+1=0.
\end{equation*}
Therefore the product (\ref{prod}) is equal to
\begin{equation*}
e^{-n\theta}(e^{2n\theta}-2e^{n\theta}\cos\omega+1)=2\cosh(n\theta)-2\cos\omega.
\end{equation*}
Using this relation in (\ref{cosh}) with $\theta=\argcosh(d-\sum_{m=1}^{d-1}\cos(2\pi k\gamma_m/\beta))$ and $\omega=2\pi k/\beta$, we have
\begin{equation}
\label{tauCirc}
\tau(C^{1,\gamma_1n,\ldots,\gamma_{d-1}n}_{\beta n})=\frac{n}{\beta}\prod_{k=1}^{\beta-1}\bigg(2\cosh(n\argcosh(d-\sum_{m=1}^{d-1}\cos(2\pi k\gamma_m/\beta)))-2\cos(2\pi k/\beta)\bigg).
\end{equation}
The theorem then follows by expressing the formula in terms of the eigenvalues on $C^{\gamma_1,\ldots,\gamma_{d-1}}_\beta$ and from the relation $\argcosh{x}=\log(x+\sqrt{x^2-1})$ for $x\geqslant1$.
\end{proof}
\begin{example}
This formula reproves Theorems $4$, $5$, $6$, $8$ and corrects a typographical error in Theorem $7$ in \cite{zhang2005chebyshev}. For example, \cite[Theorem $5$]{zhang2005chebyshev} states that
\begin{equation*}
\tau(C^{1,n}_{3n})=\frac{n}{3}\left[(\sqrt{7/4}+\sqrt{3/4})^{2n}+(\sqrt{7/4}-\sqrt{3/4})^{2n}+1\right]^2
\end{equation*}
which is a particular case of the formula with $d=2$, $\gamma_1=1$, $\beta=3$ and $\mu_k=2-2\cos(2\pi k/3)$, $k=1,2$, being the non-zero eigenvalues on the cycle $C^1_3$. As another example, \cite[Theorem $8$]{zhang2005chebyshev} states that
\begin{align*}
\tau(C^{1,2n,3n}_{6n})&=\frac{n}{6}\left[(\sqrt{11/4}+\sqrt{7/4})^{2n}+(\sqrt{11/4}-\sqrt{7/4})^{2n}-1\right]^2\left[(\sqrt{2}+1)^n+(\sqrt{2}-1)^n\right]^2\\
&\ \ \ \times\left[(\sqrt{7/4}+\sqrt{3/4})^{2n}+(\sqrt{7/4}-\sqrt{3/4})^{2n}+1\right]^2
\end{align*}
which is a particular case with $d=3$, $\gamma_1=2$, $\gamma_2=3$, $\beta=6$ and $\mu_k=4-2\cos(2\pi k/3)-2\cos(\pi k)$, $k=1,\ldots,5$, being the non-zero eigenvalues on the circulant graph $C^{2,3}_6$.
\end{example}
\begin{remark}
We emphasize that the circulant graph $C^{1,\gamma_1n,\ldots,\gamma_{d-1}n}_{\beta n}$ consists of $n$ copies of $C^{\gamma_1,\ldots,\gamma_{d-1}}_\beta$ which are embedded in the cycle $C^1_{\beta n}$. This explains the eigenvalues on $C^{\gamma_1,\ldots,\gamma_{d-1}}_\beta$ appearing in the formula. 
\end{remark}
\section{Spanning trees in discrete tori}
In this section we establish a formula for the number of spanning trees in the discrete torus $\mathbb{Z}^d/\Lambda\mathbb{Z}^d$ where $\Lambda=\textnormal{diag}(\alpha_1,\ldots,\alpha_{d-1},n)$ is a diagonal matrix with positive integer coefficients. The eigenvalues of the Laplacian on $\mathbb{Z}^d/\Lambda\mathbb{Z}^d$ are given by
\begin{align*}
\{\lambda_k\}_{k=0,1,\ldots,\det(\Lambda)-1}=\{2d-2\sum_{i=1}^{d-1}\cos(2\pi &k_i/\alpha_i)-2\cos(2\pi k'/n):\\
&0\leqslant k_i\leqslant\alpha_i-1,i=1,\ldots,d-1\textnormal{ and }0\leqslant k'\leqslant n-1\}.
\end{align*}
The formula given in the following theorem is interesting when $n$ is larger than $\det(A)$. It improves the asymptotic result given in \cite[Example $4.4.3$]{louis2013asymptotics}.
\begin{theorem}
Let $A=\textnormal{diag}(\alpha_1,\ldots,\alpha_{d-1})$ and $\{\mu_k\}_k=\{2(d-1)-2\sum_{i=1}^{d-1}\cos(2\pi k_i/\alpha_i):0\leqslant k_i\leqslant\alpha_i-1,i=1,\ldots,d-1\}$, $k=1,\ldots,\det(A)-1$, be the non-zero eigenvalues of the Laplacian on $\mathbb{Z}^{d-1}/A\mathbb{Z}^{d-1}$. For all $n\in\mathbb{N}_{n\geqslant1}$, the number of spanning trees in the discrete torus $\mathbb{Z}^d/\Lambda\mathbb{Z}^d$ is given by
\begin{equation*}
\tau(\mathbb{Z}^d/\Lambda\mathbb{Z}^d)=\frac{n}{\det(A)}\prod_{k=1}^{\det(A)-1}\left(\big(\mu_k/2+1+\sqrt{\mu_k^2/4+\mu_k}\big)^n+\big(\mu_k/2+1-\sqrt{\mu_k^2/4+\mu_k}\big)^n-2\right).
\end{equation*}
\end{theorem}
\begin{proof}
From the matrix tree theorem, we have
\begin{align*}
\tau(\mathbb{Z}^d/\Lambda\mathbb{Z}^d)&=\frac{1}{\det(A)n}\prod_{i=1}^{d-1}\prod_{\substack{k_i=0\\ \mathclap{(k_1,\ldots,k_{d-1},k')\neq0}}}^{\alpha_i-1}\prod_{k'=0}^{n-1}\left(2d-2\sum_{i=1}^{d-1}\cos(2\pi k_i/\alpha_i)-2\cos(2\pi k'/n)\right)\\
&=\frac{n}{\det(A)}\prod_{i=1}^{d-1}\prod_{\substack{k_i=0\\ \mathclap{(k_1,\ldots,k_{d-1})\neq0}}}^{\alpha_i-1}\prod_{k'=0}^{n-1}\left(2\cosh(\argcosh(d-\sum_{i=1}^{d-1}\cos(2\pi k_i/\alpha_i)))-2\cos(2\pi k'/n)\right)\\
&=\frac{n}{\det(A)}\prod_{i=1}^{d-1}\prod_{\substack{k_i=0\\ \mathclap{(k_1,\ldots,k_{d-1})\neq0}}}^{\alpha_i-1}\left(2\cosh(n\argcosh(d-\sum_{i=1}^{d-1}\cos(2\pi k_i/\alpha_i)))-2\right)
\end{align*}
where the second equality comes from equation (\ref{cycle}) and the third equality comes from the same trick as in the proof of Theorem \ref{ThCirc},
\begin{equation*}
\prod_{k=0}^{n-1}\left(2\cosh\theta-2\cos(2\pi k/n)\right)=2\cosh(n\theta)-2.
\end{equation*}
The theorem then follows by expressing the formula in terms of the eigenvalues on $\mathbb{Z}^{d-1}/A\mathbb{Z}^{d-1}$ and from the relation $\argcosh{x}=\log(x+\sqrt{x^2-1})$, for $x\geqslant1$.
\end{proof}
\section{Spanning tree entropy of circulant graphs}
For a sequence of regular graphs $G_n$ with vertex set $V(G_n)$, one can consider the number of spanning trees as a function of $n$. Assuming that the following limit exists
\begin{equation*}
z=\lim_{n\rightarrow\infty}\frac{\log\tau(G_n)}{\lvert V(G_n)\rvert},
\end{equation*}
it is sometimes called the associated tree entropy \cite{MR2160416}. From Theorem \ref{ThCirc}, the tree entropy for the non-fixed generated circulant graph $C^{1,\gamma_1n,\ldots,\gamma_{d-1}n}_{\beta n}$ as $n\rightarrow\infty$, denoted by $z_{NF}(\beta;\gamma_1,\ldots,\gamma_{d-1})$, is given in the following corollary.
\begin{corollary}
Let $1\leqslant\gamma_1\leqslant\ldots\leqslant\gamma_{d-1}\leqslant\lfloor\beta/2\rfloor$ and $\beta$ be positive integers. The tree entropy of the circulant graph $C^{1,\gamma_1n,\ldots,\gamma_{d-1}n}_{\beta n}$ as $n\rightarrow\infty$ is given by
\begin{align*}
z_{NF}(\beta;\gamma_1,\ldots,\gamma_{d-1})&=\frac{1}{\beta}\sum_{k=1}^{\beta-1}\argcosh(d-\sum_{m=1}^{d-1}\cos(2\pi k\gamma_m/\beta))\\
&=\int_0^\infty(e^{-t}-\frac{1}{\beta}\sum_{k=0}^{\beta-1}e^{-\mu_kt}e^{-2t}I_0(2t))\frac{dt}{t}
\end{align*}
where $\mu_k=2(d-1)-2\sum_{m=1}^{d-1}\cos(2\pi k\gamma_m/\beta)$, $k=0,1,\ldots,\beta-1$, are the eigenvalues of the Laplacian on the circulant graph $C^{\gamma_1,\ldots,\gamma_{d-1}}_\beta$, and $I_0$ the modified $I$-Bessel function of order zero.
\end{corollary}
\begin{proof}
From equation (\ref{tauCirc}), the asymptotic number of spanning trees in $C^{1,\gamma_1n,\ldots,\gamma_{d-1}n}_{\beta n}$ is given by
\begin{equation*}
\tau(C^{1,\gamma_1n,\ldots,\gamma_{d-1}n}_{\beta n})=\frac{n}{\beta}e^{n\sum_{k=1}^{\beta-1}\argcosh(d-\sum_{m=1}^{d-1}\cos(2\pi k\gamma_m/\beta))+o(1)}\quad\textrm{as }n\rightarrow\infty.
\end{equation*}
This shows the first equality. The second equality comes from \cite[Proposition $2.4$]{louis2013asymptotics} which expresses the $\argcosh$ in terms of an integral of modified $I$-Bessel function: for all $x\geqslant2$,
\begin{equation*}
\int_0^\infty(e^{-t}-e^{-xt}I_0(2t))\frac{dt}{t}=\argcosh(x/2).
\end{equation*}
\end{proof}
As mentioned in section \ref{th} the circulant graph $C^{1,\gamma_1n,\ldots,\gamma_{d-1}n}_{\beta n}$ consists of $n$ copies of $C^{\gamma_1,\ldots,\gamma_{d-1}}_\beta$ which are embedded in the cycle $C^1_{\beta n}$. This structure is reflected by the appearance of the term $\theta_{C^{\gamma_1,\ldots,\gamma_{d-1}}_\beta}(t)e^{-2t}I_0(2t)$ in the asymptotic formula, where $\theta_{C^{\gamma_1,\ldots,\gamma_{d-1}}_\beta}(t)=\sum_{k=0}^{\beta-1}e^{-\mu_kt}$ is the theta function on $C^{\gamma_1,\ldots,\gamma_{d-1}}_\beta$ and $e^{-2t}I_0(2t)$ is the typical term appearing in the asymptotics of the number of spanning trees in the cycle. Indeed, the tree entropy on the cycle is (see section $3.2$ in \cite{louis2013asymptotics})
\begin{equation*}
z_{cycle}=\int_0^\infty(e^{-t}-e^{-2t}I_0(2t))\frac{dt}{t}=0.
\end{equation*}

Consider the sequence of circulant graphs $C^{1,n,\gamma_1n,\ldots,\gamma_{d-1}n}_{\beta n}$ when $n\rightarrow\infty$ with $z_{NF}(\beta;1,\gamma_1,\ldots,\allowbreak\gamma_{d-1})$ denoting the corresponding tree entropy. In the following proposition we show that it is greater than the one of fixed generated circulant graphs.
\begin{proposition}
For all positive integers $\gamma_1,\ldots,\gamma_d$, there exists an integer $B\geqslant2$ such that for all $\beta\geqslant B$,
\begin{equation*}
z_{NF}(\beta;1,\gamma_1,\ldots,\gamma_{d-1})>z_F(1,\gamma_1,\ldots,\gamma_d)
\end{equation*}
where $z_F(1,\gamma_1,\ldots,\gamma_d)$ is the tree entropy of the fixed generated circulant graph $C^{1,\gamma_1,\ldots,\gamma_d}_n$.
\end{proposition}
\begin{proof}
By letting $\beta\rightarrow\infty$ in the corollary, the sum over the Laplacian eigenvalues converges to a Riemann integral, so that
\begin{equation*}
\lim_{\beta\rightarrow\infty}z_{NF}(\beta;1,\gamma_1,\ldots,\gamma_{d-1})=\int_0^\infty(e^{-t}-e^{-2(d+1)t}I_0(2t)I_0^{1,\gamma_1,\ldots,\gamma_{d-1}}(2t,\ldots,2t))\frac{dt}{t}
\end{equation*}
where $I_0^{1,\gamma_1,\ldots,\gamma_{d-1}}$ is the $d$-dimensional modified $I$-Bessel function of order zero defined by (see section $2.4$ in \cite{louis2013asymptotics})
\begin{equation*}
I_0^{1,\gamma_1,\ldots,\gamma_{d-1}}(2t,\ldots,2t)=\frac{1}{2\pi}\int_{-\pi}^{\pi}e^{2t(\cos{w}+\sum_{m=1}^{d-1}\cos(\gamma_m w))}dw.
\end{equation*}
It can be expressed in terms of a series of modified $I$-Bessel functions
\begin{equation*}
I_0^{1,\gamma_1,\ldots,\gamma_{d-1}}(2t,\ldots,2t)=\sum_{(k_1,\ldots,k_{d-1})\in\mathbb{Z}^{d-1}}I_{\sum_{i=1}^{d-1}\gamma_ik_i}(2t)\prod_{i=1}^{d-1}I_{k_i}(2t).
\end{equation*}
On the other hand, from \cite[Theorem $1.1$]{louis2013asymptotics}, the tree entropy of the fixed generated circulant graph $C^{1,\gamma_1,\ldots,\gamma_d}_n$ as $n\rightarrow\infty$ is given by
\begin{equation*}
z_F(1,\gamma_1,\ldots,\gamma_d)=\int_0^\infty(e^{-t}-e^{-2(d+1)t}I_0^{1,\gamma_1,\ldots,\gamma_d}(2t,\ldots,2t))\frac{dt}{t}
\end{equation*}
where
\begin{align*}
I_0^{1,\gamma_1,\ldots,\gamma_d}(2t,\ldots,2t)&=\sum_{(k_1,\ldots,k_d)\in\mathbb{Z}^d}I_{\sum_{i=1}^d\gamma_ik_i}(2t)\prod_{i=1}^dI_{k_i}(2t)\\
&>I_0(2t)\sum_{(k_1,\ldots,k_{d-1})\in\mathbb{Z}^{d-1}}I_{\sum_{i=1}^{d-1}\gamma_ik_i}(2t)\prod_{i=1}^{d-1}I_{k_i}(2t)\\
&=I_0(2t)I_0^{1,\gamma_1,\ldots,\gamma_{d-1}}(2t,\ldots,2t),\quad\forall t>0.
\end{align*}
Therefore
\begin{equation*}
\lim_{\beta\rightarrow\infty}z_{NF}(\beta;1,\gamma_1,\ldots,\gamma_{d-1})>z_F(1,\gamma_1,\ldots,\gamma_d).
\end{equation*}
\end{proof}
Related to this comparison between fixed and non-fixed generated circulant graphs one might wonder, for example in the simplest case of $C^{1,n}_{\beta n}$, how taking limits first in $\beta$ then in $n$ would compare to taking limits first in $n$ then in $\beta$. From \cite[Lemma $5$]{golin2010asymptotic} and by letting $\beta\rightarrow\infty$ in \cite[Theorem $4$]{golin2010asymptotic}, one easily sees that for all positive integers $\gamma_1,\ldots,\gamma_{d-1}$,
\begin{equation*}
\lim_{\gamma_d\rightarrow\infty}\lim_{n\rightarrow\infty}\frac{\log\tau(C^{\gamma_1,\ldots,\gamma_d}_{n})}{n}=\lim_{\beta\rightarrow\infty}\lim_{n\rightarrow\infty}\frac{\log\tau(C^{1,\gamma_1n,\ldots,\gamma_{d-1}n}_{\beta n})}{\beta n}
\end{equation*}
which by definition is
\begin{equation*}
\lim_{\gamma_d\rightarrow\infty}z_F(\gamma_1,\ldots,\gamma_d)=\lim_{\beta\rightarrow\infty}z_{NF}(\beta;\gamma_1,\ldots,\gamma_{d-1}).
\end{equation*}
In the particular case of $d=2$ it shows that the limits over $n$ and $\beta$ commute, that is,
\begin{equation*}
\lim_{\beta\rightarrow\infty}\lim_{n\rightarrow\infty}\frac{\log\tau(C^{1,n}_{\beta n})}{\beta n}=\lim_{n\rightarrow\infty}\lim_{\beta\rightarrow\infty}\frac{\log\tau(C^{1,n}_{\beta n})}{\beta n}
\end{equation*}
which does not seem obvious a priori.

\nocite{*}
\bibliographystyle{plain}
\bibliography{bibliographyC1n}

\end{document}